\documentclass{article}

\usepackage[english]{babel}

\usepackage{amsmath,amssymb,setspace}
\usepackage{amsthm}
\usepackage[english]{babel}
\usepackage{amsfonts}
\usepackage{calligra}
\usepackage[T1]{fontenc}
\usepackage{xstring}
\usepackage{graphics}
\usepackage[numbers,sort&compress]{natbib}
\usepackage{multicol}
\date{}
\usepackage{longtable}
\usepackage{float}
\usepackage[margin=0.7 in]{geometry}
\usepackage{graphicx}
\usepackage{mathtools}
\usepackage{amsmath}
\usepackage{subcaption}
\usepackage{amsmath, amssymb, setspace}
\usepackage{pgfplots}
\usepackage{amsthm}
\usepackage[english]{babel}
\usepackage{amsfonts}
\usepackage{calligra}
\usepackage[T1]{fontenc}
\usepackage{xstring}

\usepackage{amssymb}
\usepackage{amsthm}
\usepackage{amsmath}
\usepackage{lipsum}
\newtheorem{definition}{Definition}
\theoremstyle{plain}
\theoremstyle{definition}
\theoremstyle{remark}
\newtheorem{theorem}{Theorem}

\newtheorem{example}{Example}

\title{Higher Order Fractal Differential Equations}
\author{Alireza Khalili Golmankhaneh$^1$, Claude Depollier$^2$, Diana Pham $^3$\\
$^1$ Department of Physics, Urmia Branch,\\ Islamic Azad University, Urmia 63896,West Azerbaijan,  Iran\\
alirezakhalili2002@yahoo.co.in\\
$^2$ DLaboratoire d'Acoustique de l'Universit\'{e} du Maine, Universit\'{e} du Maine, France\\
Claude.Depollier@univ-lemans.fr\\
$^3$ Department of Biology, University of Texas at Arlington, Arlington, TX 76019, USA\\
npham@mavs.uta.edu\\
}

\begin{document}

\maketitle

\let\thefootnote\relax
\footnote{Corresponding author Alireza Khalili Golmankhaneh}
\begin{abstract}
This paper provides a summary of the fractal calculus framework. It presents higher-order homogeneous and nonhomogeneous linear fractal differential equations with $\alpha$-order. Solutions for these equations with constant coefficients are obtained through the method of variation of parameters and the method of undetermined coefficients. The solution space for higher $\alpha$-order linear fractal differential equations is defined, showcasing its non-integer dimensionality. The solutions to $\alpha$-order linear fractal differential equations are graphically depicted to illustrate their non-differentiability. Additionally, equations of motion governing the behavior of two masses in fractal time are proposed and solved.
\end{abstract} 

\section{Introduction}
Fractal geometry was popularized by Benoit Mandelbrot \cite{Mandelbro}. Studying structures with fractal dimensions larger than their topological dimensions is the main focus of this field \cite{falconer1999techniques,jorgensen2006analysis}. These complex fractal formations are self-similar and frequently have non-integer, even complex, dimensions \cite{Qaswet, Lapidus}. Reduced fractals, characterized by different levels of self-similarity and magnification, were introduced along with the construction of spectra for these fractals \cite{pham2022spectra}. Complexity science principles apply to natural hazards and human systems. Disruptive events arise from emergent properties in complex, non-equilibrium systems. Fractal analysis aids insights into extreme event development, enhancing risk assessment for such systems \cite{Wanliss23}. The local times for the superprocess were proven to exist if the spectral dimension $d_{s}$ of the spatial motion satisfies $d_{s}<4$ \cite{hambly2023dimension}. The numerical evaluation of double integrals over self-similar fractal sets with weakly singular integrands were extended \cite{caetano2023integral}. The rough surface was characterized using surface fractal dimensions, revealing heightened sensitivity of shear-thinning fluids to surface roughness in comparison to Newtonian and shear-thickening fluids \cite{bouchendouka2022fractal}. However, because conventional geometric measures such as length, surface area, volume, and Hausdorff measure \cite{rogers1998hausdorff} are mostly intended for standard forms, fractal analysis presents special difficulties \cite{Ewqq}. As such, it becomes difficult to apply these metrics directly to fractal analysis \cite{Barnsley, Gregory, rosenberg2020fractal, Tosatti, bishop2017fractals, Shlomo}. Scholars have utilized diverse techniques to tackle the problem of fractal analysis. These strategies include the following: fractional space and nonstandard techniques \cite{stillinger1977axiomatic}, probabilistic methods \cite{Barlow}, fractional calculus \cite{klafter2012fractional,e25071008, uchaikin2013fractional,tuncc2023solution, Trifcebook}, measure theory \cite{giona1995fractal, freiberg2002harmonic, jiang1998some, Bongiorno23, bongiorno2018derivatives, bongiorno2015fundamental, bongiorno2015integral}, and nonstandard methodologies \cite{nottale2011scale}. An extension of classical calculus, fractal calculus is a mathematical framework that permits the treatment of equations whose solutions manifest as functions with fractal characteristics, including fractal sets and curves \cite{parvate2009calculus, parvate2011calculus}. Fractal calculus is especially attractive because of its elegant and algorithmic approaches, which are superior to other methods \cite{Alireza-book,golmankhaneh2023fuzzification}. The fractal discharging model of the  battery was introduced to explore the impact of non-locality on solution behavior and how the system's current state is influenced by its past \cite{AliGolmankhanehYilmazer}. Fractal differential equations have been solved using a variety of techniques, and their stability conditions have been established \cite{golmankhaneh2019sumudu,tuncc2020stability,tuncc2023ulam,Fourier1}. Weierstrass functions and fractal interpolation functions are analyzed to show how FC is applied. When evaluated through the prism of classical calculus, these functions frequently exhibit traits of non-differentiability and non-integrability \cite{gowrisankar2021fractal}. The study of Cantor cubes and Cantor tartan has been included in the scope of fractal calculus \cite{golmankhaneh2018fractalt}, and the Laplace equation has been explicitly formulated within this framework \cite{khalili2021laplace}. Gauge integral approach has been effectively applied to $F^{\alpha}$-calculus (FC) generalization. This generalization is concerned with the integration of functions over a particular subset \cite{golmankhaneh2016fractal} of the real number line that incorporates singularities occurring within fractal sets. It was shown that fractional Brownian motion on thin Cantor-like sets can be described using non-local fractal derivatives. The relationship between the fractal Hurst exponent and the order of non-local fractal derivatives is established by the proposal \cite{golmankhaneh2021fractalBro}.
In addition, fractal stochastic differential equations are defined, and processes such as diffusion and fractional Brownian motion within fractal-structured media are classified \cite{golmankhaneh2021equilibrium,khalili2019fractalcat,khalili2019random,banchuin2022noise,golmankhaneh2018sub}.
It was developed to do local vector calculus on fractals, on fractal continua, and within fractional-dimensional spaces. It was proposed that in spaces with non-integer dimensions, two different del-operators may be defined, each operating on a vector field and a scalar field. The standard formulation of the Laplacian and the fundamental vector differential operators in fractional-dimensional space was achieved by utilizing these del-operators. Furthermore, $F^{\alpha}$-derivatives of Laplacian and vector differential operators on fractals were established \cite{balankin2023vector}. The Laplace transform and the steps approach were used to solve fractal retarded, neutral, and renewal delay differential equations with constant coefficients \cite{golmankhaneh2023initial}.
By using nonstandard analysis, fractal integral and differential forms were defined \cite{khalili2023non}. Recent physics research has introduced fractal time, characterized by self-similar properties and fractional dimension. This study explores fractal time within the framework of economic models, utilizing both local and non-local fractal Caputo derivatives \cite{faghih2023introduction,khalili2021economic, Welch-5,Vrobel,Shlesinger-6, plonka1995fractal}. The analogue geodesic and Einstein field equations, emphasizing the significance and applicability of fractal geometry were given \cite{golmankhaneh2023einstein}. The  double-size cancer in the fractal temporal dimension across various mathematical models were  explored \cite{golmankhaneh2024modeling}. The fractal local Mellin transform and fractal non-local transform as tools for solving fractal differential equations were given \cite{khalili2023fractalgg,khalili2023fractalewwq}. The relationships between the traditional Fokker-Planck Equation  and its fractal counterpart, incorporating fractal derivatives was given \cite{megias2023dynamics}. The Fokker-Planck equation for a fractal comb with dimension $\alpha$,  incorporating fractal time, was derived, and its solution was provided \cite{golmankhaneh2023fractalrede888}. A fractal analogue of mechanics, inspired by Newton, Lagrange, Hamilton, and Appell, were proposed. Fractal velocity and acceleration were defined to establish the Langevin equation on fractal curves \cite{golmankhaneh2023classical}.
The method analogues, including the separable method and integrating factor technique, for solving $\alpha$-order differential equations were studied \cite{golmankhaneh2023solving}.\\
In this work, building upon the aforementioned research, we introduce higher-order $\alpha$-fractal differential equations and present a methodology for solving them.\\
The paper is structured as follows:\\
In Section \ref{1g}, a concise review of fractal calculus is presented. Section \ref{2g} introduces higher $\alpha$-order linear fractal differential equations. In Section \ref{3g}, the homogeneous higher $\alpha$-order linear fractal differential equation is introduced and solved. Section \ref{4g} provides nonhomogeneous higher $\alpha$-order linear fractal differential equations. Homogeneous higher $\alpha$-order linear fractal differential equations with constant coefficients are solved in Section \ref{5g}. The specific solutions for nonhomogeneous higher $\alpha$-order linear fractal differential equations are obtained using the variation of parameters and the method of undetermined coefficients in Section \ref{6g}. Finally, Section \ref{7g} presents the conclusion.

\section{Overview of Fractal Calculus on Fractal Sets\label{1g}}
In this section, we present a concise overview of fractal calculus applied to fractal sets as summarized in \cite{parvate2009calculus,parvate2011calculus,Alireza-book}.
\begin{definition}
The flag function of a set $F$ and a closed interval $I$ is defined as:
\begin{equation}
  \rho(F,I)=
  \begin{cases}
    1, & \text{if } F\cap I\neq\emptyset;\\
    0, & \text{otherwise}.
  \end{cases}
\end{equation}
\end{definition}

\begin{definition}
For a fractal set $F\subset [a,b]$, a subdivision $P_{[a,b]}$ of $[a,b]$, and a given $\delta>0$, the coarse-grained mass of $F\cap [a,b]$ is defined by
\begin{equation}
  \gamma_{\delta}^{\alpha}(F,a,b)=\inf_{|P|\leq
\delta}\sum_{i=0}^{n-1}\Gamma(\alpha+1)(t_{i+1}-t_{i})^{\alpha}
\rho(F,[t_{i},t_{i+1}]),
\end{equation}
where $|P|=\max_{0\leq i\leq n-1}(t_{i+1}-t_{i})$, and $0< \alpha\leq1$.
\end{definition}

\begin{definition}
The mass function of a fractal set $F$ is defined as the limit of the coarse-grained mass as $\delta$ approaches zero:
\begin{equation}
  \gamma^{\alpha}(F,a,b)=\lim_{\delta\rightarrow0}\gamma_{\delta}^{\alpha}(F,a,b).
\end{equation}
\end{definition}

\begin{definition}
The $\gamma$-dimension of $F\cap [a,b]$ is defined as:
\begin{align}
  \dim_{\gamma}(F\cap
[a,b])&=\inf\{\alpha:\gamma^{\alpha}(F,a,b)=0\}\nonumber\\&
=\sup\{\alpha:\gamma^{\alpha}(F,a,b)=\infty\}
\end{align}
\end{definition}

\begin{definition}
The integral staircase function of order $\alpha$ for a fractal set $F$ is given by:
\begin{equation}
 S_{F}^{\alpha}(x)=
 \begin{cases}
   \gamma^{\alpha}(F,a_{0},x), & \text{if } x\geq a_{0}; \\
   - \gamma^{\alpha}(F,x,a_{0}), & \text{otherwise}.
 \end{cases}
\end{equation}
where $a_{0}$ is an arbitrary fixed real number.
\end{definition}

\begin{definition}
Let $F$ be an $\alpha$-perfect fractal set, let $f$ be a function defined on F and let $x\in F.$ The $F^{\alpha}$-derivative of $f$ at the point $x$ is defined as follows:
\begin{equation}
  D_{F}^{\alpha}f(x)=
  \begin{cases}
    \underset{ y\rightarrow
x}{F_{-}\text{lim}}~\frac{f(y)-f(x)}{S_{F}^{\alpha}(y)-S_{F}^{\alpha}(x)}, & \text{if } x\in F; \\
    0, & \text{otherwise}.
  \end{cases}
\end{equation}
if the fractal limit $F_{-}\text{lim}$ exists \cite{parvate2009calculus}.
\end{definition}

\begin{definition}
Let $I=[a,b]$.~Let $F$ be an $\alpha$-perfect fractal set such that $S^{\alpha}_F$ is finite on $I$. Let $f$ be a bounded function defined on F and let $x\in F.$ The $F^{\alpha}$-integral of $f$ on $I$ is defined as:
\begin{align}
  \int_{a}^{b}f(x)d_{F}^{\alpha}x&=\sup_{P_{[a,b]}}
\sum_{i=0}^{n-1}\inf_{x\in F\cap
I}f(x)(S_{F}^{\alpha}(x_{i+1})-S_{F}^{\alpha}(x_{i}))
\nonumber\\&=\inf_{P_{[a,b]}}
\sum_{i=0}^{n-1}\sup_{x\in F\cap
I}f(x)(S_{F}^{\alpha}(x_{i+1})-S_{F}^{\alpha}(x_{i})).
\end{align}

\end{definition}

\section{Higher $\alpha$-Order Linear Fractal Differential Equations \label{2g}}
Within this section, we present higher-order linear fractal differential equations and their corresponding solutions. An $n^{th}$ $\alpha$-order linear fractal differential equation can be expressed in the following form:
\begin{align}\label{tyyy}
  P_{0}(x)D_{F}^{n\alpha}f(x)&+P_{1}(x)D_{F}^{(n-1)\alpha}f(x)\nonumber\\&
  +...+P_{n-1}(x)
  D_{F}^{\alpha}f(x)+P_{n}(x)f(x)=G(x).
\end{align}
In cases where the functions $P_{0}, \ldots, P_{n}$ and $G(x)$ exhibit $F$-continuity over the interval $I:a<x<b$, with the additional condition that $P_{0}$ is non-zero everywhere in this interval, dividing Eq.\eqref{tyyy} by $P_{0}(x)$ results in:
\begin{align}\label{iio232o}
  L[f]&=D_{F}^{n\alpha}f(x)+p_{1}(x)D_{F}^{(n-1)\alpha}f(x)\nonumber\\&+...+
  p_{n-1}(x)D_{F}^{\alpha}f(x)+p_{n}(x)f(x)=g(x).
\end{align}
To derive a unique solution for Eq. \eqref{iio232o}, it is necessary to have $n$ initial conditions, as outlined below:
\begin{align}\label{iiinnmj}
  f(x_{0})&=f_{0},~~~D_{F}^{\alpha}f(x)|_{x=x_{0}}=D_{F}^{\alpha}f(x_{0})\nonumber\\&,...,
  D_{F}^{\alpha(n-1)}f(x)|_{x=x_{0}}=D_{F}^{\alpha(n-1)}f(x_{0}).
\end{align}
Here, $x_{0}$ represents any point within the interval $(a, b)$, and $$f_{0}, D_{F}^{\alpha}f(x_{0}),...,D_{F}^{\alpha(n-1)}f(x_{0}),$$ can take  any  real constant numbers.
\begin{definition}
The solution space for higher $\alpha$-order linear fractal differential equations has a dimensionality of $n\alpha$.
\end{definition}

\section{The Homogeneous Higher $\alpha$-Order Linear Fractal Differential Equation \label{3g}}
The $n^{th}$ $\alpha$-order fractal differential equation in its homogeneous form can be expressed as:
\begin{align}\label{ioooploikju}
  L(f)&=D_{F}^{n\alpha}f(x)+p_{1}D_{F}^{(n-1)\alpha}f(x)+...+
  p_{n-1}D_{F}^{\alpha}f(x)+p_{n}(t)f(x)=0,\nonumber\\&
  =f^{n\alpha}(x)+p_{1}f^{(n-1)\alpha}(x)+...+
  p_{n-1}f^{\alpha}(x)+p_{n}(t)f(x)=0.
\end{align}
If the functions $f_{1}, f_{2}, \ldots, f_{n}$ constitute solutions to Eq.\eqref{ioooploikju}, it consequently follows that:
\begin{equation}\label{yhubvdfgd}
  f(x)=c_{1}f_{1}(x)+c_{2}f_{2}(x)+...+c_{n}f_{n}(x),
\end{equation}
 where $c_{1}, c_{2}, \ldots, c_{n}$ represent arbitrary constants, any solution of Eq.\eqref{yhubvdfgd} can be expressed as:
\begin{align}\label{ijuyy}
  c_{1}f_{1}(x_{0})+\cdots+c_{n}f_{n}(x_{0})&=f(x_{0})=f_{0}\nonumber\\
  c_{1}f_{1}^{\alpha}(x_{0})+\cdots+c_{n}f_{n}^{\alpha}(x_{0})&=
  f^{\alpha}(x_{0})=f^{\alpha}_{0}\nonumber\\
  &\vdots\nonumber\\
  c_{1}f_{1}^{(n-1)\alpha}(x_{0})+\cdots+c_{n}f_{n}^{(n-1)\alpha}(x_{0})&=
  f^{(n-1)\alpha}(x_{0})=f^{(n-1)\alpha}_{0},
\end{align}
here $x\in (a,b)$ and $D_{F}^{n\alpha}f(x)=f^{n\alpha}(x)$.

For arbitrary values of $f_{0}, f^{\alpha}{0}, \ldots, f^{(n-1)\alpha}{0}$, a necessary and sufficient condition for the existence of a solution to the system of algebraic equations in \eqref{ijuyy} is that the Wronskian be:
\begin{equation}\label{trf23}
  W[f_{0},f_{1},\cdots,f_{n}]=\begin{vmatrix}
f_{1} & f_{2} & \cdots &  f_{n}\\
f_{1}^{\alpha} & f_{2}^{\alpha} & \cdots &  f_{n}^{\alpha} \\
\vdots & \vdots & \cdots &  \vdots \\
f_{1}^{(n-1)\alpha} & f_{2}^{(n-1)\alpha} & \cdots &  f_{n}^{(n-1)\alpha}
\end{vmatrix}
\end{equation}
To ensure that the Wronskian $W[f_{0}, f_{1}, \ldots, f_{n}]$ is non-zero at $x=x_{0}$, where $x_{0}\in (a,b)$, it is necessary and sufficient that $W[f_{0}, f_{1}, \ldots, f_{n}]$ is non-zero at every point within the interval $(a,b)$. Consequently, we establish the following theorem:
\begin{theorem}\label{rtee}
If the functions $p_{1}, p_{2}, \ldots, p_{n}$ are $F$-continuous on the open interval $(a,b)$, and $f_{1}, f_{2}, \ldots, f_{n}$ are solutions to equation \eqref{ioooploikju}, with $W[f_{1}, f_{2}, \ldots, f_{n}]\neq 0$ for at least one point in $(a,b)$, then every solution of equation \eqref{ioooploikju} can be represented as a linear combination of the solutions $f_{1}, f_{2}, \ldots, f_{n}$.
  \end{theorem}
  \begin{proof}
  Consider the equation
\begin{equation}\label{ionhhplm}
k_{1}f_{1}(x)+k_{2}f_{2}(x)+\ldots+k_{n}f_{n}(x)=0.
\end{equation}
By repeatedly differentiating \eqref{ionhhplm}, we obtain an additional $n-1$ equations:

\begin{align}\label{ioknbvcf}
k_{1}f_{1}^{\alpha}(x)+k_{2}f_{2}^{\alpha}(x)+\ldots+k_{n}f_{n}^{\alpha}(x)&=0
\nonumber\\
&\vdots \nonumber\\
k_{1}f_{1}^{(n-1)\alpha}(x)+k_{2}f_{2}^{(n-1)\alpha}(x)
+\ldots+k_{n}f_{n}^{(n-1)\alpha}(x)&=0
\end{align}
These equations form a system of algebraic equations for $n$ unknowns $k_{1}, \ldots, k_{n}$. The determinant of the coefficients for this system is the Wronskian $W[f_{1},f_{2}, \ldots, f_{n}]$ of $f_{1}, f_{2}, \ldots, f_{n}$.
\end{proof}
\begin{theorem}
If $f_{1}(x), \ldots, f_{n}(x)$ constitute a fundamental set of solutions for the homogeneous $n^{th}$ order linear fractal differential equation:
\begin{equation}\label{ioa}
  L[f]=f^{n\alpha}(x)+p_{1}(x)f^{(n-1)\alpha}(x)+\cdots p_{n-1}(x)f^{\alpha}(x)+p_{n}(x)f(x)=0,
\end{equation}
defined on an interval $(a,b)$, then $f_{1}(x), \ldots, f_{n}(x)$ are linearly independent on $(a,b)$. Conversely, if $f_{1}(x), \ldots, f_{n}(x)$ are linearly independent solutions of equation \eqref{ioa} on $(a,b)$, then they form a fundamental set of solutions on $(a,b)$.
 \end{theorem}
\begin{proof}
To establish this theorem, let's initially assume that $f_{1}(x), \ldots, f_{n}(x)$ constitute a fundamental set of solutions for the homogeneous fractal differential equation \eqref{ioa} on $(a,b)$. Consequently, the Wronskian $Wf_{1}, \ldots, f_{n}\neq 0$ for every $x\in (a,b)$. Therefore, the system \eqref{ionhhplm} and \eqref{ioknbvcf} has only the solution $k_{1}=\ldots=k_{n}=0$ for every $x$ in $(a,b)$. Consequently, $f_{1}(x), \ldots, f_{n}(x)$ must be linearly independent on $(a,b)$.

To demonstrate the converse, let $f_{1}(x), \ldots, f_{n}(x)$ be linearly independent on $(a,b)$. In order to show that they form a fundamental set of solutions, it is necessary to demonstrate that their Wronskian is never zero in $(a,b)$. Suppose, for the sake of contradiction, that this is not true, indicating that there exists at least one point $x_{0}$ where the Wronskian is zero. At this point, the system \eqref{ionhhplm} and \eqref{ioknbvcf} has a non-zero solution denoted by $k^{}_{1}, \ldots, k^{}_{n}$. Forming the linear combination:

\begin{equation}\label{ooil}
\psi(t)=k^{}{1}f{1}(x)+\ldots+k^{}{n}f{n}(x).
\end{equation}

Then $f=\psi(t)$ satisfies the initial value problem:

\begin{equation}\label{ereww}
L[f]=0,~~~~f(x_{0})=0,~f^{\alpha}(x_{0})=0,\ldots,f^{(n-1)\alpha}(x_{0})=0.
\end{equation}

The function $\psi$ satisfies the fractal differential equation because it is a linear combination of solutions, and it satisfies the initial conditions because these are just the equations in the system \eqref{ionhhplm} and \eqref{ioknbvcf} evaluated at $x_{0}$. However, the function $f(x)=0$ for all $x\in (a,b)$ is also a solution of this initial value problem, and by Theorem \ref{rtee}, the solution to the initial value problem \eqref{ereww} is unique. Thus, $\psi(t)=0$ for all $x$ in $(a,b)$. Consequently, $f_{1}(x), \ldots, f_{n}$ are linearly dependent on $(a,b)$, which is a contradiction. Hence, the assumption that there is a point where the Wronskian is zero is untenable. Hence, the Wronskian is never zero on $(a,b)$, as was to be proved.
\end{proof}
\section{The Nonhomogeneous Higher $\alpha$-Order  Linear Fractal Differential Equation \label{4g}}
In this section, we introduce nonhomogeneous $n^{th}$ $\alpha$-order fractal differential equations along with their solutions.\\
Consider the nonhomogeneous $n^{th}$ $\alpha$-order fractal differential equation:
\begin{equation}\label{ijnb741}
  L[f]=f^{n\alpha}(x)+p_{1}(x)f^{(n-1)\alpha}(x)+...+p_{n}(x)f(x)=g(x).
\end{equation}
If $F_{1}$ and $F_{2}$ are any two solutions, it immediately follows from the linearity of the fractal operator $L$ that:
\begin{equation}\label{tyrqplm}
  L[F_{1}-F_{2}](x)=L[F_{1}](x)-L[F_{2}](x)=g(x)-g(x)=0.
\end{equation}
Hence, the difference of any two solutions of the nonhomogeneous fractal differential Eq.\eqref{ijnb741} is a solution of the homogeneous fractal differential equation Eq.\eqref{ioooploikju}. Since any solution of the nonhomogeneous equation can be expressed as a linear combination of a fundamental set of solutions $f_{1}, f_{2}, \ldots, f_{n}$, it follows that any solution of the nonhomogeneous fractal differential equation Eq.\eqref{ijnb741} can be written as:
\begin{equation}\label{ijnmkiuyh}
 f(x)=c_{1}f_{1}(x)+c_{2}f_{2}(x)+...+c_{n}f_{n}(x)+F(x),
\end{equation}
where $F(x)$ is some particular solution of Eq.\eqref{ijnb741}. The linear combination Eq.\eqref{ijnmkiuyh} is called the general solution of the nonhomogeneous Eq.\eqref{ijnb741}.
\section{Homogeneous Higher $\alpha$-Order Linear Fractal Differential Equations with Constant Coefficients \label{5g} }
Consider the $n^{th}$ $\alpha$-order linear homogeneous fractal differential equation:
\begin{equation}\label{iijnbvc4}
  L[f]=a_{0}f^{n\alpha}(x)+a_{1}f^{(n-1)\alpha}(x)+...+a_{n}f^{\alpha}(x)+a_{n}f(x)=0.
\end{equation}
Here, $a_{0},a_{1},...,a_{n}$ are real constants, and $a_{0}\neq0$. It is natural to anticipate that $f=\exp(rS_{F}^{\alpha}(x))$ is a solution of Eq.\eqref{iijnbvc4} for suitable values of $r$. Namely:
 \begin{align}\label{ioofcv}
   L[\exp(rS_{F}^{\alpha}(x))]&=\exp(rS_{F}^{\alpha}(x))
   (a_{0}r^{n}+a_{1}r^{n-1}+...+a_{n-1}r+a_{n}\nonumber\\
   &=\exp(rS_{F}^{\alpha}(x))Z(r)
 \end{align}
 Here
 \begin{equation}\label{rftrde}
   Z(r)=a_{0}r^{n}+a_{1}r^{n-1}+...+a_{n-1}r+a_{n}.
 \end{equation}
  For the values of $r$ for which $Z(r)=0$, it follows that $L[\exp(rS_{F}^{\alpha}(x))]=0$, and $f=\exp(rS_{F}^{\alpha}(x))$ is a solution of Eq.\eqref{iijnbvc4}. The equation $Z(r)=0$ is the characteristic equation of Eq.\eqref{iijnbvc4}. Since $a_{0}\neq0$, $Z(r)$ has $n$ roots, namely, $r_{1},r_{2},\ldots,r_{n}$. Some roots may be equal, and some may be complex; therefore, we can write:
\begin{equation}\label{rtre}
  Z(r)=a_{0}(r-r_{1})(r-r_{2})...(r-r_{n}).
\end{equation}
We have the following cases:
\begin{enumerate}
  \item Real and Unequal Roots:
  If the roots of the characteristic equation are real and distinct, we have $n$ distinct solutions:
  \begin{equation}\label{iiiii}
    \exp(r_{1}S_{F}^{\alpha}(x)),
    \exp(r_{2}S_{F}^{\alpha}(x)),...,\exp(r_{n}S_{F}^{\alpha}(x))
  \end{equation}
  If these functions are linearly independent, the solution of Eq.\eqref{iijnbvc4} is:
  \begin{equation}\label{yhubhj}
    f=c_{1} \exp(r_{1}S_{F}^{\alpha}(x))+c_{2} \exp(r_{2}S_{F}^{\alpha}(x))+...+c_{n} \exp(r_{n}S_{F}^{\alpha}(x))
  \end{equation}
  \item Complex roots:
  Since the coefficients $a_{0}, a_{1}, \ldots,a_{n}$ are real numbers, if the characteristic equation has complex roots that occur in conjugate pairs $\lambda\pm i\mu$.
  \item  Repeated roots: If a root of $Z(r)=0$, say $r=r_{1}$, has multiplicity $s\leq n$, then the solutions are:
      \begin{align}\label{ioopnju}
        &\exp(r_{1}S_{F}^{\alpha}(x)), S_{F}^{\alpha}(x)\exp(r_{1}S_{F}^{\alpha}(x)),
        S_{F}^{\alpha}(x)^{2}\exp(r_{1}S_{F}^{\alpha}(x))\nonumber\\&,...,
        S_{F}^{\alpha}(x)^{s-1}\exp(r_{1}S_{F}^{\alpha}(x)),
      \end{align}
      If a complex root $\lambda+i\mu$ is repeated $s$ times, then we have $2s$ complex-valued solutions:
      \begin{align}\label{iokj}
      &  \exp((\lambda+i\mu)S_{F}^{\alpha}(x)),
        S_{F}^{\alpha}(x)\exp((\lambda+i\mu)S_{F}^{\alpha}(x)),...,\nonumber\\&
        S_{F}^{\alpha}(x)^{s-1}\exp((\lambda+i\mu)S_{F}^{\alpha}(x)),\nonumber\\&
        \exp((\lambda-i\mu)S_{F}^{\alpha}(x)),
        S_{F}^{\alpha}(x)\exp((\lambda-i\mu)S_{F}^{\alpha}(x)),...,\nonumber\\&
        S_{F}^{\alpha}(x)^{s-1}\exp((\lambda-i\mu)S_{F}^{\alpha}(x))
      \end{align}
      which leads to $2s$ real-valued solutions:
      \begin{align}\label{iooppmki}
       & \exp(\lambda S_{F}^{\alpha}(x))\cos(\mu S_{F}^{\alpha}(x)),\exp(\lambda S_{F}^{\alpha}(x))\sin(\mu S_{F}^{\alpha}(x)),\nonumber\\& S_{F}^{\alpha}(x))\exp(\lambda S_{F}^{\alpha}(x))\cos(\mu S_{F}^{\alpha}(x)), S_{F}^{\alpha}(x))\exp(\lambda S_{F}^{\alpha}(x))\sin(\mu S_{F}^{\alpha}(x))\nonumber\\&,...,S_{F}^{\alpha}(x))^{s-1}\exp(\lambda S_{F}^{\alpha}(x))\cos(\mu S_{F}^{\alpha}(x))\nonumber\\&,S_{F}^{\alpha}(x))^{s-1}\exp(\lambda S_{F}^{\alpha}(x))\sin(\mu S_{F}^{\alpha}(x))
      \end{align}
\end{enumerate}

\begin{example}
Consider the $4^{th}$ $\alpha$-order fractal differential equation:
\begin{equation}\label{eewsa}
  f^{4\alpha}(x)+2f^{2\alpha}(x)+f(x)=0.
\end{equation}
Since the roots of the characteristic equation are $r_{1}=i$ and $r_{2}=-i$ with multiplicity $2$, the general solution of Eq.\eqref{eewsa} is given by:
\begin{align}\label{deerr}
  &f(x)=c_{1}\cos(S_{F}^{\alpha}(x))+c_{2}\sin(S_{F}^{\alpha}(x))+c_{3}
  S_{F}^{\alpha}(x)\cos(S_{F}^{\alpha}(x))\nonumber\\&+c_{4}
  S_{F}^{\alpha}(x)\sin(S_{F}^{\alpha}(x)),
\end{align}

In Figure \ref{Esre}, Eq.\eqref{deerr} is plotted for the case of $c_{1}=c_{2}=c_{3}=c_{4}$.
\begin{figure}
  \centering
  \includegraphics[scale=0.5]{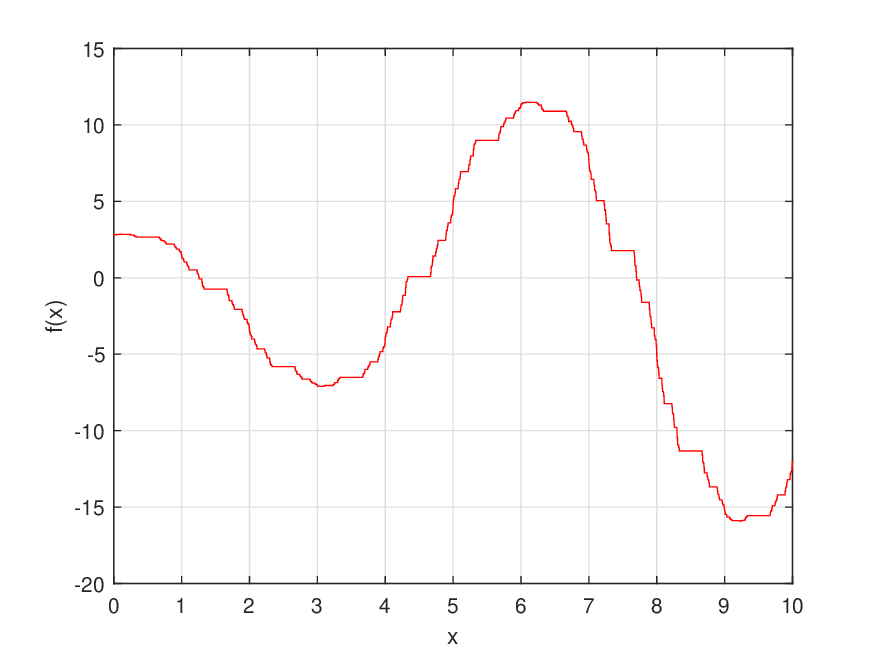}
  \caption{The graph of Eq.\eqref{deerr}.}\label{Esre}
\end{figure}
\end{example}
\section{Solving Nonhomogeneous Higher $\alpha$-Order Linear Fractal Differential Equations \label{6g}}
In this section, we discuss techniques for determining a specific solution to nonhomogeneous higher $\alpha$-order linear fractal differential equations. Specifically, we explore two methods: the method of undetermined coefficients and the method of variation of parameters.

\subsection{The Method of Undetermined Coefficients}
Consider an $n^{th}$ $\alpha$-order linear fractal differential equation with constant coefficients given by:
\begin{equation}\label{yttgfser}
  L[f]=a_{0}f^{n\alpha}(x)+a_{1}f^{(n-1)\alpha}(x)+...+a_{n-1}f^{\alpha}(x)+
  a_{n}f(x)=g(x),
\end{equation}
If $g(x)$ can be expressed as a sum of polynomials, exponentials, sines, cosines, or products of such functions, the method of undetermined coefficients is employed. To find a particular solution $F(x)$, a suitable combination of these functions is chosen.
\begin{example}
Consider a $3^{rd}$ $\alpha$-order linear fractal differential equation:
\begin{equation}\label{iiiio}
  f^{3\alpha}(x)-3f^{2\alpha}(x)+3f^{\alpha}(x)-f(x)=4\exp(S_{F}^{\alpha}(x)),
\end{equation}
The solution to the characteristic equation reveals a root $r=1$ with multiplicity $3$. Thus, the general solution to the homogeneous part of Eq.\eqref{iiiio} is given by:
\begin{equation}\label{ppkjuhy}
  f_{c}(x)=c_{1}\exp(S_{F}^{\alpha}(x))+
  c_{2}S_{F}^{\alpha}(x)\exp(S_{F}^{\alpha}(x))+
  c_{3}S_{F}^{\alpha}(x)^{2}\exp(S_{F}^{\alpha}(x)).
\end{equation}
To find a particular solution $F(x)$ for Eq.\eqref{iiiio}, let's suppose that $F(x)=A\exp(S_{F}^{\alpha}(x))$. By considering the form in Eq.\eqref{ppkjuhy}, we modify $F(x)$ to $F(x)=AS_{F}^{\alpha}(x)^{3}\exp(S_{F}^{\alpha}(x))$, where $A$ is a determined coefficient. To find the correct value of $A$, we $F^{\alpha}$-differentiate $F(x)$ three times and substitute for $f$ and its fractal derivatives in Eq.\eqref{iiiio}, leading to:
\begin{equation}\label{iokmnbv}
  6A\exp(S_{F}^{\alpha}(x))=4\exp(S_{F}^{\alpha}(x))
\end{equation}
Solving for $A$, we get $A=\frac{2}{3}$. Thus, the particular solution is:
\begin{equation}\label{iomnbv}
  F(x)=\frac{2}{3}S_{F}^{\alpha}(x)^{3}\exp(S_{F}^{\alpha}(x)).
\end{equation}
The general solution to the nonhomogeneous fractal differential equation Eq.\eqref{iiiio} is then given by:
\begin{align}\label{uuuuu5}
  f(x)=c_{1}\exp(S_{F}^{\alpha}(x))+
  c_{2}S_{F}^{\alpha}(x)\exp(S_{F}^{\alpha}(x))+&
  c_{3}S_{F}^{\alpha}(x)^{2}\exp(S_{F}^{\alpha}(x))\nonumber\\&+
  \frac{2}{3}S_{F}^{\alpha}(x)^{3}\exp(S_{F}^{\alpha}(x)).
\end{align}
In Figure \ref{Espplre}, Eq.\eqref{uuuuu5} has been graphed for the specific case of $c_{1}=c_{2}=c_{3}=1$.
\begin{figure}
  \centering
  \includegraphics[scale=0.5]{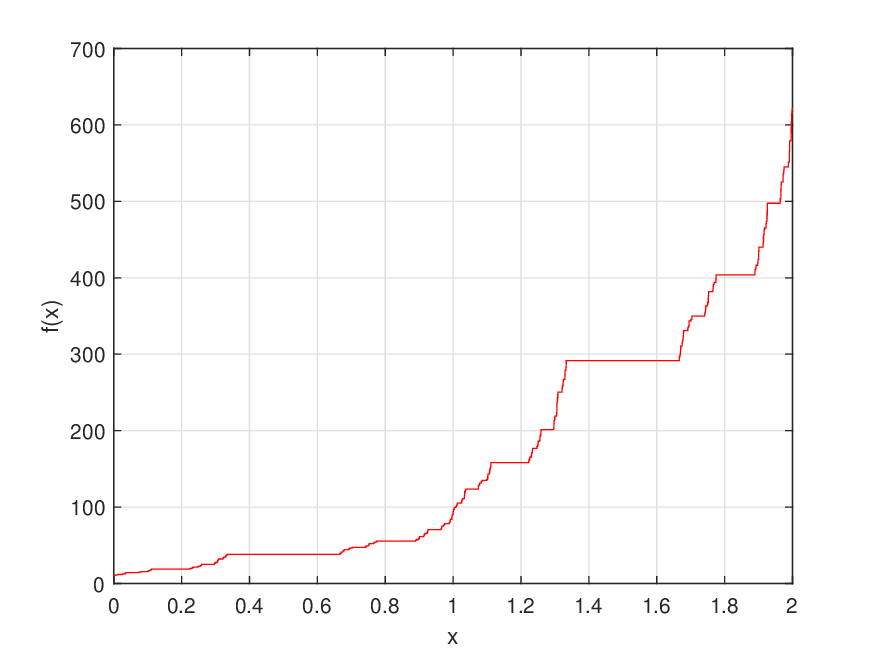}
  \caption{The graph of Eq.\eqref{uuuuu5}.}\label{Espplre}
\end{figure}
\end{example}

\subsection{The Method of Variation of Parameters}
The variation of parameters method for determining a particular solution of a nonhomogeneous $n^{th}$ $\alpha$-order fractal linear differential equation is as follows. Consider the equation:
\begin{equation}\label{ijnyh}
  L[f]=f^{\alpha n}(x)+p_{1}(x)f^{\alpha (n-1)}(x)+...+p_{n-1}(x)f^{\alpha}(x)+
  p_{n}(x)f(x)=g(x),
\end{equation}
Suppose that $f_{1}, f_{2}, ..., f_{n}$ are fundamental solutions of the homogeneous part of Eq.\eqref{ijnyh}. The general solution is then expressed as:
\begin{equation}\label{iijjnbbgt}
  f_{c}(x)=c_{1}f_{1}(x)+c_{2}f_{2}(x)+...+c_{n}f_{n}(x).
\end{equation}
To find a particular solution $F(x)$ of Eq.\eqref{ijnyh}, assume:
\begin{equation}\label{reqaw}
  F(x)=u_{1}(x)y_{1}(x)+u_{2}(x)y_{2}(x)+...+u_{n}(x)y_{n}(x)
\end{equation}
Now, find the functions $u_{1}, u_{2}, ..., u_{n}$. By taking the fractal derivative of Eq.\eqref{reqaw}, we get:
\begin{align}\label{ijreqaw}
  F^{\alpha}(x)&=(u_{1}(x)f_{1}^{\alpha}(x)+u_{2}(x)f_{2}^{\alpha}(x)+...+
  u_{n}(x)f_{n}^{\alpha}(x))\nonumber\\&+
  (u_{1}^{\alpha}(x)f_{1}(x)+u_{2}^{\alpha}(x)f_{2}(x)+...+u_{n}^{\alpha}(x)f_{n}(x))
\end{align}
 The first condition to find $u_{1}, u_{2}, ..., u_{n}$ is:
\begin{equation}\label{rdewse}
  u_{1}^{\alpha}(x)f_{1}(x)+u_{2}^{\alpha}(x)f_{2}(x)+...+u_{n}^{\alpha}(x)f_{n}(x)=0
\end{equation}
Continue this process by calculating successive fractal derivatives $F^{2\alpha}, ..., F^{\alpha(n-1)}$. After each $F^{\alpha}$-differentiation, set the sum of terms involving the fractal derivative of $u_{1}, u_{2}, ..., u_{n}$ equal to zero. This leads to $n-2$ further conditions:
\begin{equation}\label{ioj}
  u_{1}^{\alpha}(x)f_{1}^{m\alpha}(x)+
  u_{2}^{\alpha}(x)f_{2}^{m\alpha}(x)+...+u_{n}^{\alpha}(x)f_{n}^{m\alpha}(x)=0,~~~
  m=1,2,...,n-2.
\end{equation}
As a result, the expression for $F^{2\alpha}, ..., F^{(n-1)\alpha}$ reduces to:
\begin{equation}\label{tygh}
  F^{m\alpha}(x)=u_{1}(x)f_{1}^{m \alpha }(x)+u_{2}(x)f_{2}^{ m \alpha}(x)+...+u_{n}(x)f_{n}^{m \alpha}(x),
\end{equation}
 Finally, impose the condition that $F$ must be a solution of Eq.\eqref{ijnyh}. The $n^{th}$ $\alpha$-order differential of $F$ is:
 \begin{align}\label{iojplm}
   F^{n\alpha}(x)&=(u_{1}(x)f_{1}^{n \alpha }(x)+u_{2}(x)f_{2}^{ n \alpha}(x)+...+u_{n}(x)f_{n}^{n \alpha}(x))\nonumber\\&+
   (u_{1}^{\alpha}(x)f_{1}^{(n-1) \alpha }(x)+u_{2}^{\alpha}(x)f_{2}^{ (n-1) \alpha}(x)+...+u_{n}^{\alpha}(x)f_{n}^{(n-1) \alpha}(x))
 \end{align}
 Now, substitute $F(x)$ and its fractal derivatives into Eq.\eqref{ijnyh}, and group the terms involving each of the functions $f_{1}, f_{2}, ..., f_{n}$ and their fractal derivatives. Most terms drop out because each $f_{1}, f_{2}, ..., f_{n}$ satisfies $L[f_{i}]=0$ for $i=1,2,...,n$. The remaining terms yield the relation:
\begin{equation}\label{ijnhyubgt}
  u_{1}^{\alpha}(x)f_{1}^{(n-1) \alpha }(x)+u_{2}^{\alpha}(x)f_{2}^{(n-1) \alpha }(x)+...+u_{n}^{\alpha}(x)f_{n}^{(n-1) \alpha }(x)=g(x)
\end{equation}
Eqs. \eqref{ijnhyubgt}, \eqref{rdewse}, and the $n-2$ Eqs. \eqref{ioj} provide $n$ simultaneous linear nonhomogeneous algebraic equations for $u_{1}^{\alpha}, u_{2}^{\alpha}, ..., u_{n}^{\alpha}$:
\begin{align}\label{rqaqqa}
  f_{1}u_{1}^{\alpha}+f_{2}u_{2}^{\alpha}+...+f_{n}u_{n}^{\alpha}&=0\nonumber\\
  f_{1}^{\alpha}u_{1}^{\alpha}+f_{2}^{\alpha}u_{2}^{\alpha}+...
  +f_{n}^{\alpha}u_{n}^{\alpha}&=0\nonumber\\
  f_{1}^{2\alpha}u_{1}^{\alpha}+f_{2}^{2\alpha}u_{2}^{\alpha}+...
  +f_{n}^{2\alpha}u_{n}^{\alpha}&=0\nonumber\\
  &\vdots\nonumber\\
  f_{1}^{(n-1)\alpha}u_{1}^{\alpha}+f_{2}^{(n-1)\alpha}u_{2}^{\alpha}+...
  +f_{n}^{(n-1)\alpha}u_{n}^{\alpha}&=g
\end{align}
Using Cramer's rule, solve the system of equations \eqref{rqaqqa} to find the solutions in the form:
\begin{equation}\label{zsaaqq}
  u_{m}^{\alpha}=\frac{g(x)W_{m}(x)}{W(x)},~~~m=1,2,...,n,
\end{equation}
Here, $W(x)=Wf_{1}, f_{2}, ..., f_{n}$ and $W_{m}$ is the determinant obtained from $W$ by replacing the $m^{th}$ column with the transpose of column $(0,0,...,0,1)$. A particular solution of Eq.\eqref{ijnyh} is then given by:
\begin{equation}\label{rdeeess}
  F(x)=\sum_{m=1}^{n}f_{m}(x)\int_{x_{0}}^{x}\frac{g(s)W_{m}(s)}{W(s)}d_{F}^{\alpha}s,
\end{equation}
where $x_{0}$ is arbitrary.
\begin{example}
Consider the $3^{rd}$ $\alpha$-order linear nonhomogeneous fractal differential equation:
\begin{equation}\label{dswww}
  f^{3\alpha}(x)- f^{2\alpha}(x)- f^{\alpha}(x)+f=g(x)
\end{equation}
To find the general solution of Eq.\eqref{dswww}, first, solve the $3^{rd}$ $\alpha$-order linear homogeneous fractal differential equation. Its general solution is:
\begin{equation}\label{ddddddd}
  f_{c}(x)=c_{1}\exp(S_{F}^{\alpha}(x))+
  c_{2}S_{F}^{\alpha}(x)\exp(S_{F}^{\alpha}(x))
  +c_{3}\exp(-S_{F}^{\alpha}(x)),
\end{equation}
where $\exp(S_{F}^{\alpha}(x)),~S_{F}^{\alpha}(x)
\exp(S_{F}^{\alpha}(x)),~\exp(-S_{F}^{\alpha}(x))$ are fundamental independent solutions of the homogeneous equation. To find a particular solution, use Eq.\eqref{rdeeess}. The result is:
\begin{align}\label{rszasdeeess}
  F(x)&=\exp(S_{F}^{\alpha}(x))\int_{x_{0}}^{x}\frac{g(t)(-1-2t)}
  {4\exp(S_{F}^{\alpha}(x))}d_{F}^{\alpha}t
  \nonumber\\&+S_{F}^{\alpha}(x)\exp(S_{F}^{\alpha}(x))\int_{x_{0}}^{x}\frac{g(t)(2)}
  {4\exp(S_{F}^{\alpha}(x))}d_{F}^{\alpha}t\nonumber\\&
  +\exp(-S_{F}^{\alpha}(x))\int_{x_{0}}^{x}\frac{g(s)\exp(2S_{F}^{\alpha}(t))}
  {4\exp(S_{F}^{\alpha}(x))}d_{F}^{\alpha}t
\end{align}
Further simplification yields:
\begin{align}\label{rszzaaqasdeeess}
  F(x)&=\frac{1}{4}
  \int_{x_{0}}^{x}(\exp(S_{F}^{\alpha}(x)-S_{F}^{\alpha}(t))
  (-1+2(S_{F}^{\alpha}(x)-S_{F}^{\alpha}(t)))\nonumber\\&+
  \exp(-(S_{F}^{\alpha}(x)-S_{F}^{\alpha}(t))))g(t)d_{F}^{\alpha}t
\end{align}
For the choice $g(t)=S_{F}^{\alpha}(t)^{-2}\exp(S_{F}^{\alpha}(t))$, the particular solution is:
\begin{equation}\label{oloii}
  F(x)=-S_{F}^{\alpha}(x)\exp(S_{F}^{\alpha}(x))\ln|S_{F}^{\alpha}(x)|
\end{equation}
\end{example}

\begin{example}
Consider the spring-mass system comprising two unit masses connected to springs with spring constants $k_{1}$ and $k_{2}$, as depicted in Figure \ref{wetry}.
\begin{figure}[H]
  \centering
  \begin{tikzpicture}

\def\springlength{4}  
\def\massone{1}
\def\masstwo{0.5}

\draw (2.2,0) -- (2.2,0.5) node[midway,below] {};

\draw (2.2,0.2) -- (2.7,0.2);

\draw[decorate,decoration={coil,amplitude=6pt,segment length=4pt}] (2.7,0.2) -- (2.7+\springlength,0.2);
\node[above] at (2.7+\springlength/2,0.5) {$k_1$};

\draw[fill=blue!30] (2.7+\springlength, 0) rectangle ++(0.5*\massone,0.5);
\node[above] at (2.95+\springlength,0.5) {$m_1$};

\draw[decorate,decoration={coil,amplitude=6pt,segment length=4pt}] (2.7+\springlength+0.5*\massone,0.2) -- (2.7+2*\springlength+0.5*\massone,0.2);
\node[above] at (2.7+1.5*\springlength+0.25*\massone,0.5) {$k_2$};

\draw[fill=red!40] (2.7+2*\springlength+0.5*\massone, 0) rectangle ++(0.5*\masstwo,0.5);
\node[above] at (2.95+2*\springlength+0.25*\masstwo,0.5) {$m_2$};
\end{tikzpicture}
  \caption{Graph of the spring-mass system.}\label{wetry}
\end{figure}
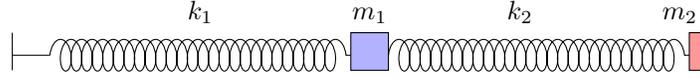
The equations of motion governing the behavior of two masses in fractal time are described as follows:
\begin{align}
  u_{1}^{2\alpha}+(k_{1}+k_{2})u_{1}&=k_{2}u_{2},~~~~\textmd{for}~~~\textmd{mass}~~~m_{1}\label{eerdesw}\\
  u_{2}^{2\alpha}+k_{2}u_{2}&=k_{2}u_{1},~~~~\textmd{for}~~\textmd{~mass}~~~m_{2}\label{eerdesw2}
\end{align}
Here, $u_{1}$ and $u_{2}$ represent the displacements of the masses from their respective equilibrium positions. Substituting $u_{2}$ from Eq.\eqref{eerdesw} into Eq.\eqref{eerdesw2} yields the following equation:
\begin{equation}\label{iokml}
  u_{1}^{4\alpha}+(k_{1}+2k_{2})u_{1}^{2\alpha}+k_{1}k_{2}u_{1}=0,
\end{equation}
The solution to Eq.\ref{iokml} is expressed as:
\begin{align}\label{iop415}
  u_{1c}(x)=c_{1}\cos(S_{F}^{\alpha}(x))+&
  c_{2}\sin(S_{F}^{\alpha}(x))\nonumber\\&+c_{3}\cos(\sqrt{6}S_{F}^{\alpha}(x))
  +c_{4}\sin(\sqrt{6}S_{F}^{\alpha}(x))
\end{align}
By choosing the initial conditions specified in Eq.\eqref{ttttttt1}, the solutions to the given initial value problem are:
\begin{equation}\label{ttttttt1}
 u_{1}(0)=0,~~~u_{1}^{\alpha}(0)=0,~~~u_{2}(0)=2,~~~u_{2}^{\alpha}(0)=0.
\end{equation}
Alternatively, by selecting the initial conditions in Eq.\eqref{ttttttt2}, the solutions to the differential equation described by Eq.\eqref{iokml} are:
\begin{equation}\label{ionhiopkiujy}
   u_{1}(x)=\cos(S_{F}^{\alpha}(x)),~~~~u_{2}(x)=2\cos(S_{F}^{\alpha}(x))
\end{equation}
In the case where the specified initial conditions are:
\begin{equation}\label{ttttttt2}
 u_{1}(0)=-2,~~~u_{1}^{\alpha}(0)=0,~~~u_{2}(0)=1,~~~u_{2}^{\alpha}(0)=0.
\end{equation}
The solutions to the differential equation described by Eq. \eqref{iokml}, along with the provided initial value problem, are:
\begin{align}\label{88ionhy}
   u_{1}(x)=-2\cos(\sqrt{6}S_{F}^{\alpha}(x)),~~~~
   u_{2}(x)=\cos(\sqrt{6}S_{F}^{\alpha}(x)).
\end{align}
\begin{figure}
  \centering
  \includegraphics[scale=0.5]{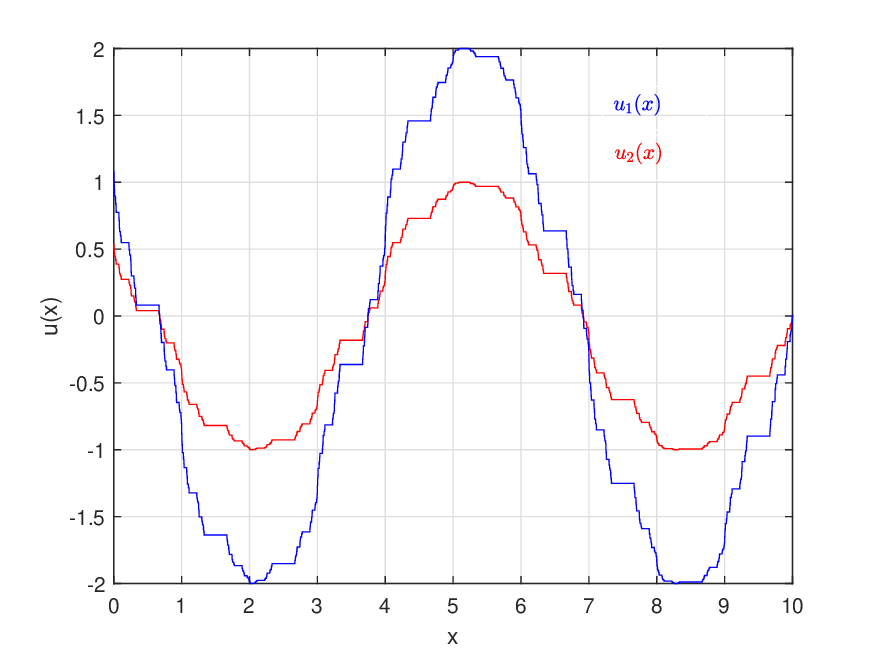}
  \caption{The graph of Eq.\eqref{ionhiopkiujy}.}\label{Eseesepplre}
\end{figure}
\begin{figure}
  \centering
  \includegraphics[scale=0.5]{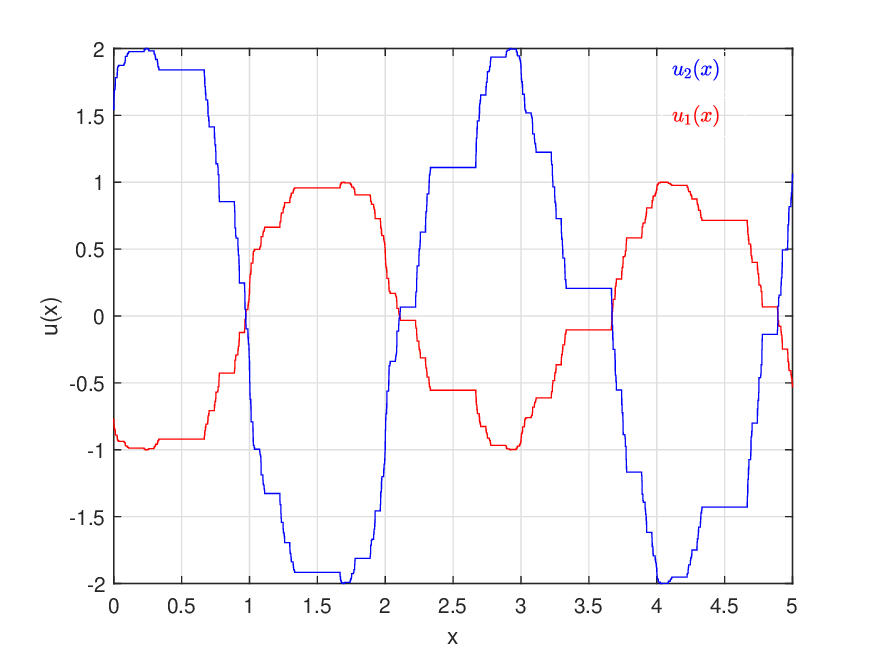}
  \caption{The graph of Eq.\eqref{88ionhy}.}\label{qweeEspplre}
\end{figure}
The graphs illustrating the solutions provided in equations \eqref{ionhiopkiujy} and \eqref{88ionhy} can be observed in Figures \ref{Eseesepplre} and \ref{qweeEspplre}.
\end{example}


\section{Conclusion \label{7g}}
This paper has delved into the realm of $n$th $\alpha$-order fractal homogeneous and nonhomogeneous differential equations, providing solutions for both. Specifically, we tackled higher or $n$th $\alpha$-order fractal homogeneous equations with constant coefficients, employing the Method of variation of parameters and the method of undetermined coefficients to find explicit solutions for higher nonhomogeneous fractal differential equations.
Furthermore, we applied these methods to explore the equations of motion governing a system of two masses in fractal time. The obtained solutions were graphically represented, offering insights into the dynamic behavior of the system.
In essence, this work contributes to the understanding and solution techniques for fractal differential equations, providing a foundation for further exploration and applications in diverse scientific domains.\\
\textbf{Declaration of Competing Interest:}\\
The authors declare that they have no known competing financial interests or personal relationships that could have appeared to influence the work reported in this paper.


\bibliographystyle{elsarticle-num}
\bibliography{Refrancesma10}

\begin{thebibliography}{10}
\expandafter\ifx\csname url\endcsname\relax
  \def\url#1{\texttt{#1}}\fi
\expandafter\ifx\csname urlprefix\endcsname\relax\def\urlprefix{URL }\fi
\expandafter\ifx\csname href\endcsname\relax
  \def\href#1#2{#2} \def\path#1{#1}\fi

\bibitem{Mandelbro}
B.~B. Mandelbrot, The Fractal Geometry of Nature, WH freeman New York, 1982.

\bibitem{falconer1999techniques}
K.~Falconer, Fractal Geometry: Mathematical Foundations and Applications, John Wiley \& Sons, 2004.

\bibitem{jorgensen2006analysis}
P.~E. Jorgensen, Analysis and probability: wavelets, signals, fractals, Vol. 234, Springer Science \& Business Media, 2006.

\bibitem{Qaswet}
P.~R. Massopust, Fractal Functions, Fractal Surfaces, and Wavelets, Academic Press, 2017.

\bibitem{Lapidus}
M.~L. Lapidus, G.~Radunovi{\'{c}}, D.~{\v{Z}}ubrini{\'{c}}, Fractal Zeta Functions and Fractal Drums, Springer International Publishing, 2017.

\bibitem{pham2022spectra}
D.~T. Pham, Z.~E. Musielak, Spectra of reduced fractals and their applications in biology, Fractal and Fractional 7~(1) (2022) 28.

\bibitem{Wanliss23}
J.~Wanliss, \href{https://doi.org/10.5772/intechopen.1002058}{Fractal methods in nonequilibrium systems}, in: D.~S.-A. Ouadfeul (Ed.), Fractal Analysis - Applications and Updates, IntechOpen, Rijeka, 2023, Ch.~2.
\newblock \href {https://doi.org/10.5772/intechopen.1002058} {\path{doi:10.5772/intechopen.1002058}}.
\newline\urlprefix\url{https://doi.org/10.5772/intechopen.1002058}

\bibitem{hambly2023dimension}
B.~Hambly, P.~Koepernik, Dimension results and local times for superdiffusions on fractals, Stochastic Process. Appl. 158 (2023) 377--417.

\bibitem{caetano2023integral}
A.~Caetano, S.~Chandler-Wilde, X.~Claeys, A.~Gibbs, D.~Hewett, A.~Moiola, Integral equation methods for acoustic scattering by fractals, arXiv preprint arXiv:2309.02184 (2023).

\bibitem{bouchendouka2022fractal}
A.~Bouchendouka, Z.~E.~A. Fellah, Z.~Larbi, Z.~Louna, E.~Ogam, M.~Fellah, C.~Depollier, Fractal analysis of a non-newtonian fluid flow in a rough-walled pipe, Materials 15~(10) (2022) 3700.

\bibitem{rogers1998hausdorff}
C.~A. Rogers, Hausdorff Measures, Cambridge University Press, 1998.

\bibitem{Ewqq}
N.~Lesmoir-Gordon, B.~Rood, Introducing Fractal Geometry, Icon Books, 2000.

\bibitem{Barnsley}
M.~F. Barnsley, Fractals Everywhere, Academic Press, 2014.

\bibitem{Gregory}
T.~G. Dewey, Fractals in Molecular Biophysics, Oxford University Press, 1998.

\bibitem{rosenberg2020fractal}
E.~Rosenberg, Fractal dimensions of networks, Vol.~1, Springer, 2020.

\bibitem{Tosatti}
L.~Pietronero, E.~Tosatti (Eds.), Fractals in Physics, Elsevier, 1986.

\bibitem{bishop2017fractals}
C.~J. Bishop, Y.~Peres, Fractals in probability and analysis, Vol. 162, Cambridge University Press, 2017.

\bibitem{Shlomo}
A.~Bunde, S.~Havlin, Fractals in Science, Springer, 2013.

\bibitem{stillinger1977axiomatic}
F.~H. Stillinger, Axiomatic basis for spaces with noninteger dimension, J. Math. Phys. 18~(6) (1977) 1224--1234.

\bibitem{Barlow}
M.~T. Barlow, E.~A. Perkins, Brownian motion on the sierpinski gasket, Probab. Theory Rel. 79~(4) (1988) 543--623.

\bibitem{klafter2012fractional}
J.~Klafter, S.~C. Lim, R.~Metzler, Fractional dynamics: recent advances, World Scientific, 2012.

\bibitem{e25071008}
A.~Deppman, E.~Meg\'{\i}as, R.~Pasechnik, Fractal derivatives, fractional derivatives and q-deformed calculus, Entropy 25~(7) (2023).

\bibitem{uchaikin2013fractional}
V.~V. Uchaikin, Fractional Derivatives for Physicists and Engineers, Vol.~2, Springer, 2013.

\bibitem{tuncc2023solution}
O.~Tun{\c{c}}, C.~Tun{\c{c}}, Solution estimates to caputo proportional fractional derivative delay integro-differential equations, Revista de la Real Academia de Ciencias Exactas, F{\'\i}sicas y Naturales. Serie A. Matem{\'a}ticas 117~(1) (2023) 12.

\bibitem{Trifcebook}
T.~Sandev, {\v{Z}}.~Tomovski, Fractional Equations and Models, Springer International Publishing, 2019.

\bibitem{giona1995fractal}
M.~Giona, Fractal calculus on [0, 1], Chaos Solit. Fractals 5~(6) (1995) 987--1000.

\bibitem{freiberg2002harmonic}
U.~Freiberg, M.~Z{\"a}hle, Harmonic calculus on fractals-a measure geometric approach {I}, Potential Anal. 16~(3) (2002) 265--277.

\bibitem{jiang1998some}
H.~Jiang, W.~Su, Some fundamental results of calculus on fractal sets, Commun. Nonlinear Sci. Numer. Simul. 3~(1) (1998) 22--26.

\bibitem{Bongiorno23}
D.~Bongiorno, Derivation and Integration on a Fractal Subset of the Real Line, IntechOpen, 2023, Ch.~7.
\newblock \href {https://doi.org/10.5772/intechopen.1001895} {\path{doi:10.5772/intechopen.1001895}}.

\bibitem{bongiorno2018derivatives}
D.~Bongiorno, Derivatives not first return integrable on a fractal set, Ric. di Mat. 67~(2) (2018) 597--604.

\bibitem{bongiorno2015fundamental}
D.~Bongiorno, G.~Corrao, On the fundamental theorem of calculus for fractal sets, Fractals 23~(02) (2015) 1550008.

\bibitem{bongiorno2015integral}
D.~Bongiorno, G.~Corrao, An integral on a complete metric measure space, Real Anal. Exch. 40~(1) (2015) 157--178.

\bibitem{nottale2011scale}
L.~Nottale, Scale relativity and fractal space-time: a new approach to unifying relativity and quantum mechanics, World Scientific, 2011.

\bibitem{parvate2009calculus}
A.~Parvate, A.~D. Gangal, Calculus on fractal subsets of real line-{I}: Formulation, Fractals 17~(01) (2009) 53--81.

\bibitem{parvate2011calculus}
A.~Parvate, S.~Satin, A.~Gangal, Calculus on fractal curves in $\mathbb{R}^{n}$, Fractals 19~(01) (2011) 15--27.

\bibitem{Alireza-book}
A.~K. Golmankhaneh, Fractal Calculus and its Applications, World Scientific, 2022.

\bibitem{golmankhaneh2023fuzzification}
A.~K. Golmankhaneh, K.~Welch, C.~Serpa, P.~E. J{\o}rgensen, Fuzzification of fractal calculus, arXiv preprint arXiv:2302.07641 (2023).

\bibitem{AliGolmankhanehYilmazer}
K.~A. Kamal, A.~K. Golmankhaneh, R.~Yilmazer, Battery discharging model on fractal time sets, Int. J. Nonlin. Sci. Num. 24~(1) (2021) 71--80.

\bibitem{golmankhaneh2019sumudu}
A.~K. Golmankhaneh, C.~Tun{\c{c}}, Sumudu transform in fractal calculus, Appl. Math. Comput. 350 (2019) 386--401.

\bibitem{tuncc2020stability}
C.~Tun{\c{c}}, A.~K. Golmankhaneh, On stability of a class of second alpha-order fractal differential equations, AIMS Mathematics 5~(3) (2020) 2126--2142.

\bibitem{tuncc2023ulam}
O.~Tun{\c{c}}, C.~Tun{\c{c}}, Ulam stabilities of nonlinear iterative integro-differential equations, Revista de la Real Academia de Ciencias Exactas, F{\'\i}sicas y Naturales. Serie A. Matem{\'a}ticas 117~(3) (2023) 118.

\bibitem{Fourier1}
A.~K. Golmankhaneh, K.~Ali, R.~Yilmazer, M.~Kaabar, Local fractal {F}ourier transform and applications, Comput. Methods Differ. Equ. 10~(3) (2021) 595--607.

\bibitem{gowrisankar2021fractal}
A.~Gowrisankar, A.~Khalili~Golmankhaneh, C.~Serpa, Fractal calculus on fractal interpolation functions, Fractal Fract. 5~(4) (2021) 157.

\bibitem{golmankhaneh2018fractalt}
A.~K. Golmankhaneh, A.~Fernandez, Fractal calculus of functions on cantor tartan spaces, Fractal Fract. 2~(4) (2018) 30.

\bibitem{khalili2021laplace}
A.~K. Golmankhaneh, S.~M. Nia, Laplace equations on the fractal cubes and casimir effect, Eur. Phys. J. Special Topics 230~(21) (2021) 3895--3900.

\bibitem{golmankhaneh2016fractal}
A.~K. Golmankhaneh, D.~Baleanu, Fractal calculus involving gauge function, Commun. Nonlinear Sci. Numer. Simul. 37 (2016) 125--130.

\bibitem{golmankhaneh2021fractalBro}
A.~K. Golmankhaneh, R.~T. Sibatov, Fractal stochastic processes on thin {C}antor-like sets, Mathematics 9~(6) (2021) 613.

\bibitem{golmankhaneh2021equilibrium}
A.~K. Golmankhaneh, K.~Welch, Equilibrium and non-equilibrium statistical mechanics with generalized fractal derivatives: A review, Mod. Phys. Lett. A 36~(14) (2021) 2140002.

\bibitem{khalili2019fractalcat}
A.~K. Golmankhaneh, C.~Cattani, Fractal logistic equation, Fractal Fract. 3~(3) (2019) 41.

\bibitem{khalili2019random}
A.~K. Golmankhaneh, A.~Fernandez, Random variables and stable distributions on fractal {C}antor sets, Fractal Fract. 3~(2) (2019) 31.

\bibitem{banchuin2022noise}
R.~Banchuin, Noise analysis of electrical circuits on fractal set, Compel- Int. J. Comput. Math. Electr. Electron. Eng. 41~(5) (2022) 1464--1490.

\bibitem{golmankhaneh2018sub}
A.~K. Golmankhaneh, A.~S. Balankin, Sub-and super-diffusion on {C}antor sets: Beyond the paradox, Phys. Lett. A. 382~(14) (2018) 960--967.

\bibitem{balankin2023vector}
A.~S. Balankin, B.~Mena, Vector differential operators in a fractional dimensional space, on fractals, and in fractal continua, Chaos Solit. Fractals 168 (2023) 113203.

\bibitem{golmankhaneh2023initial}
A.~K. Golmankhaneh, I.~Tejado, H.~Sevli, J.~E.~N. Vald{\'e}s, On initial value problems of fractal delay equations, Appl. Math. Comput. 449 (2023) 127980.

\bibitem{khalili2023non}
A.~K. Golmankhaneh, K.~Welch, C.~Serpa, P.~E. J{\o}rgensen, Non-standard analysis for fractal calculus, J. Anal. 31 (2023) 1895--1916.

\bibitem{faghih2023introduction}
N.~Faghih, An introduction to time and fractals: Perspectives in economics, entrepreneurship, and management, in: Time and Fractals: Perspectives in Economics, Entrepreneurship, and Management, Springer, 2023, pp. 1--11.

\bibitem{khalili2021economic}
A.~K. Golmankhaneh, K.~K~Ali, R.~Yilmazer, M.~KA~Kaabar, Economic models involving time fractal, J. Math. Model. Financ. 1~(1) (2021) 159--178.

\bibitem{Welch-5}
K.~Welch, A Fractal Topology of Time: Deepening into Timelessness, Fox Finding Press, 2020.

\bibitem{Vrobel}
S.~Vrobel, Fractal Time, World Scientific, 2011.

\bibitem{Shlesinger-6}
M.~F. Shlesinger, Fractal time in condensed matter, Annu. Rev. Phys. Chem. 39~(1) (1988) 269--290.

\bibitem{plonka1995fractal}
A.~Plonka, Fractal time rate processes in polymer systems, Radiat. Phys. Chem. 45~(1) (1995) 67--70.

\bibitem{golmankhaneh2023einstein}
A.~K. Golmankhaneh, P.~E. J{\o}rgensen, A.~M. Schlichtinger, Einstein field equations extended to fractal manifolds: A fractal perspective, J. Geom. Phys. 196 (2023) 105081.

\bibitem{golmankhaneh2024modeling}
A.~K. Golmankhaneh, S.~Tun{\c{c}}, A.~M. Schlichtinger, D.~M. Asanza, A.~K. Golmankhaneh, Modeling tumor growth using fractal calculus: Insights into tumor dynamics, Biosystems 235 (2024) 105071.

\bibitem{khalili2023fractalgg}
A.~K. Golmankhaneh, K.~Welch, C.~Serpa, P.~E. J{\o}rgensen, Fractal mellin transform and non-local derivatives, Georgian Math. J. (2023).
\newblock \href {https://doi.org/10.1515/gmj-2023-2094} {\path{doi:10.1515/gmj-2023-2094}}.

\bibitem{khalili2023fractalewwq}
A.~K. Golmankhaneh, K.~Welch, C.~Serpa, R.~Rodr{\'\i}guez-L{\'o}pez, Fractal laplace transform: analyzing fractal curves, J. Anal. (2023) 1--27\href {https://doi.org/10.1007/s41478-023-00677-1} {\path{doi:10.1007/s41478-023-00677-1}}.

\bibitem{megias2023dynamics}
E.~Meg{\'\i}as, A.~K. Golmankhaneh, A.~Deppman, Dynamics in fractal spaces, Phys. Lett. B (2023) 138370.

\bibitem{golmankhaneh2023fractalrede888}
A.~K. Golmankhaneh, L.~A.~O. Ontiveros, Fractal calculus approach to diffusion on fractal combs, Chaos Solit. Fractals 175 (2023) 114021.

\bibitem{golmankhaneh2023classical}
A.~K. Golmankhaneh, K.~Welch, C.~Tun{\c{c}}, Y.~S. Gasimov, Classical mechanics on fractal curves, Eur. Phys. J. Spec. Top 232 (2023) 991--999.

\bibitem{golmankhaneh2023solving}
A.~K. Golmankhaneh, D.~Bongiorno, Solving and applying fractal differential equations: Exploring fractal calculus in theory and practice, arXiv preprint arXiv:2310.16515 (2023).

\end{thebibliography}

\end{document}